\documentclass{amsart}
\usepackage{amsfonts}

\newcommand{\e}{\varepsilon}

\newtheorem{theorem}{Theorem}[section]

\newtheorem{problem}[theorem]{Problem}
\newtheorem{proposition}[theorem]{Proposition}
\newtheorem{corollary}[theorem]{Corollary}
\newtheorem{claim}[theorem]{Claim}

\begin{document}
\title{On local convexity of nonlinear mappings between Banach spaces}
\author{Iryna Banakh}
\address{Ya.Pidstryhach Institute for Applied Problems of Mechanics and
Mathematics of National Academy of Sciences, Lviv, Ukraine}
\email{ibanakh@yahoo.com}
\author{Taras Banakh}
\address{Ivan Franko National University of Lviv (Ukraine), and Jan
Kochanowski University, Kielce (Poland)}
\email{t.o.banakh@gmail.com}
\author{Anatolij Plichko}
\address{Cracow University of Technology, Krak\'ow, Poland}
\email{aplichko@pk.edu.pl}
\author{Anatoliy Prykarpatsky}
\address{AGH University of Science and Technology, Krak\'{o}w 30059 Poland \
}
\email{pryk.anat@ua.fm, prykanat@cybergal.com}

\begin{abstract}
We find conditions for a smooth nonlinear map $f:U\rightarrow V$ between
open subsets of Hilbert or Banach spaces to be locally convex in the sense
that for some $c$ and each positive $\varepsilon<c$ the image $%
f(B_\varepsilon(x))$ of each $\varepsilon$-ball $B_\varepsilon(x)\subset U$
is convex. We give a lower bound on $c$ via the second order Lipschitz
constant $\mathrm{Lip}_2(f)$, the Lipschitz-open constant $\mathrm{Lip}_o(f)$
of $f$, and the 2-convexity number $\mathrm{conv}_2(X)$ of the Banach space $%
X$.
\end{abstract}
\maketitle

\section*{\protect\medskip Introduction}

The local convexity of nonlinear mappings of Banach spaces is important in
many branches of applied mathematics \cite{Aug,BP,KZ,PBPP,Pr1,Pr,SPS}, in
particular, in the theory of nonlinear differential-operator equations,
optimization and control theory etc. Locally convex maps appear naturally in
various problems of Fixed Point Theory \cite{Ge1,Ge2,Go} and Nonlinear Analysis \cite{Ho,Li,Nir,Sch}.

Let $X,Y$ be Banach spaces. A map $f:U\rightarrow Y$ defined on an open
subset $U\subset X$ is called \emph{locally convex} at a point $x\in U$ if
there is a positive constant $c>0$ such that for each positive $%
\varepsilon\le c $ and each point $x\in U$ with $B_\varepsilon(x)\subset U$
the image $f(B_\varepsilon(x))$ is convex. Here $B_\varepsilon(x)=\{y\in
X:\|x-y\|<\varepsilon\}$ stands for the open $\varepsilon$-ball centered at $%
x$. The local convexity of $f$ at $x$ can be expressed via the \emph{local
convexity radius}
\begin{equation*}
\mathrm{lcr}_x(f)=\sup\big\{c\in[0,+\infty):\forall \varepsilon\le
c\;\forall x\in U\mbox{ with }B_\varepsilon(x)\subset U\mbox{ the set
$f(B_\e(x))$ is convex}\big\}.
\end{equation*}
It follows that $f$ is locally convex at $x\in U$ if and only if $\mathrm{lcr%
}_x(f)>0$.

A map $f:U\to Y$ is defined to be

\begin{itemize}
\item \emph{locally convex} if $f$ is locally convex at each point $x\in U$;

\item \emph{uniformly locally convex} if its \emph{local convexity radius} $%
\mathrm{lcr}(f)=\inf\limits_{x\in U}\mathrm{lcr}_x(f)$ is not equal to zero.
\end{itemize}

For example, if a homeomorphism $f:U\to V$ between open subsets $U\subset X$%
, $V\subset Y$ with $f(0)=0\in U$ is norm convex in the sense that
\begin{equation*}
\Big\|f\Big(\frac{x+x^{\prime }}{2}\Big)\Big\|\le \frac12\big(%
\|f(x)\|+\|f(x^{\prime })\|\big)\mbox{  \ for all $x, x'\in f(U)$},
\end{equation*}
then the inverse map $f^{-1}$ is locally convex at the point $y=0$. In
particular, if $Y$ is a Banach lattice with the order $\le $ and a
homeomorphism $f:U\to V$ is Jensen convex, i.e.
\begin{equation*}
f\Big(\frac{x+x^{\prime }}{2}\Big)\le \frac12\big(f(x)+f(x^{\prime })\big)
\end{equation*}
for all $x, x^{\prime }\in U$, then the inverse map $f^{-1}$ is locally
convex at the point $y=0$. \smallskip

In this paper we find some conditions on a map $f:U\to Y$ guaranteeing that $%
f$ uniformly locally convex, and give a lower bound on the local convexity
radius $\mathrm{lcr}(f)$ of $f$. This bound depends on the second order
Lipschitz constant $\mathrm{Lip}_2(f)$ of $f$, the Lipschitz-open constant $%
\mathrm{Lip}_o(f)$ of $f$, and the 2-convexity number $\mathrm{conv}_2(X)$
of the Banach space $X$.

\section{Banach spaces with modulus of convexity of power type 2}

The \emph{modulus of convexity} of a Banach space $X$ is the function $%
\delta_X:[0,2]\to[0,1]$ assigning to each number $t\ge 0$ the real number
\begin{equation*}
\delta_X(t)=\inf\big\{1-\big\|\tfrac{x+y}2\big\|:x,y\in S_X,\;\|x-y\|\ge t%
\big\},
\end{equation*}
where $S_X=\{x\in X:\|x\|=1\}$ is the unit sphere of the Banach space $X$.
By \cite[p.60]{LT2}, the modulus of convexity can be equivalently defined as
\begin{equation*}
\delta_X(t)=\inf\big\{1-\big\|\tfrac{x+y}2\big\|:x,y\in B_X,\;\|x-y\|\ge t%
\big\},
\end{equation*}
where $B_X=\{x\in X:\|x\|\le1\}$ is the closed unit ball of $X$.

Any Hilbert space $E$ of dimension $\dim (E)>1$ has modulus of convexity
\begin{equation*}
\frac18t^2\le \delta _{E}(t)=1-\sqrt{1-(t/2)^2}\leq \frac{1}{4}t^{2}.
\end{equation*}
By \cite[p63]{LT2} or \cite{GH}, $\delta _{X}(t)\leq \delta _{E}(t)\leq
\frac{1}{4}t^{2}$ for each Banach space $X$.

Following \cite[p.63]{LT2}, \cite[p.154]{DGZ}, we say that the Banach space $%
X$ has \emph{modulus of convexity of power type $p$} if there is a constant $%
L>0$ such that $\delta _{X}(t)\geq L\cdot t^{p}$ for all $t\in \lbrack 0,2]$%
. It follows from $L\,t^{p}\leq \delta _{X}(t)\leq \frac{1}{8}t^{2}$ that $%
p\geq 2 $. Hilbert spaces have modulus of convexity of power type 2. Many
examples of Banach spaces with modulus of convexity of power type 2 can be
found in \cite[\S 1.e]{LT2}, \cite[Ch.V]{DGZ}, \cite{BGHV}, \cite{LPT}, and
\cite{HMZ}. In particular, the class of Banach spaces with modulus of
convexity of power type 2 includes the Banach spaces $L_p$ for $1<p\le 2$,
and reflexive subspaces of the Banach space $L_1$. By \cite{GH}, a Banach
space $X$ has modulus of convexity of power type 2 if and only if for any
sequences $(x_n)_{n\in\omega}$ and $(y_n)_{n\in\omega}$ in $X$ the
convergence $2(\|x_n\|^2+\|y_n\|^2)-\|x_n+y_n\|^2\to0$ implies $%
\|x_n-y_n\|\to 0$.

For a Banach space $X$ consider the constant
\begin{equation*}
\mathrm{conv}_2(X)=\inf\Big\{\frac{1-\|\frac{x+y}2\|}{\|x-y\|^2}:x,y\in B_X,
\;\;x\ne y\Big\}\ge0
\end{equation*}%
called the \emph{2-convexity number} of $X$ and observe that $\mathrm{conv}%
_2(X)>0$ if and only if $X$ has modulus of convexity of power type 2. It
follows from \cite[p.63]{LT2} or \cite{GH} that
\begin{equation*}
0\le\mathrm{conv}_2(X)\le \mathrm{conv}_2(\ell_2)=\frac18
\end{equation*}%
for each Banach space $X$.

\section{Moduli of smoothness of maps of Banach spaces}

In this section we recall known information \cite[\S 2.7]{VL} on the moduli
of smoothness $\omega _{n}(f,t)$ of a function $f:U\rightarrow Y$ defined on
a subset $U\subset X$ of a Banach space $X$ with values in a Banach space $Y$.

The \emph{$n$-th modulus of smoothness} of $f$ is defined as
\begin{equation*}
\omega_n(f,t)=\sup\{\|\Delta^n_h(f,x)\|:h\in X,\;\|h\|\le
t,\;[x,x+nh]\subset U\}
\end{equation*}
where
\begin{equation*}
\Delta^n_h(f,x)=\sum_{k=0}^n(-1)^k\binom{n}{k}f(x+kh)
\end{equation*}%
is the $n$-th difference of $f$.

In particular,
\begin{equation*}
\begin{aligned} \omega _{1}(f,t)&=\sup \{\Vert f(x+h)-f(x)\Vert :\; \Vert
h\Vert \leq t,\;[x,x+h]\subset U\}\mbox{ and }\\ \omega _{2}(f,t)&=\sup
\{\Vert f(x+h)-2f(x)+f(x-h)\Vert :\;\Vert h\Vert \leq t,\;[x-h,x+h]\subset
U\}. \end{aligned}
\end{equation*}
Here $[x,y]=\{tx+(1-t)y:t\in \lbrack 0,1]\}$ stands for the segment
connecting two points $x,y\in X$.

The constants
\begin{equation*}
\mathrm{Lip}_1(f)=\sup_{t>0}\frac{\omega_1(f,t)}{t}\mbox{ \ \ and \ \ }%
\mathrm{Lip}_2(f)=\sup_{t>0}\frac{\omega_2(f,t)}{t^2}
\end{equation*}%
are called the \emph{Lipschitz constant} and the \emph{second order
Lipschitz constant} of $f$, respectively. \smallskip

A function $f:U\to Y$ is called (\emph{second order}) \emph{Lipschitz} if
its (second order) Lipschitz constant $\mathrm{Lip}_1(f)$ (resp. $\mathrm{Lip%
}_2(f)$) is finite. The second order Lipschitz property of a weakly
G\^ateaux differentiable function $f$ can be deduced from the Lipschitz
property of its derivative $f^{\prime }$. \smallskip

Let us recall \cite[p.154]{Deimling} that a function $f:U\rightarrow Y$ is \emph{weakly G\^{a}teaux differentiable} at a point $x\in U$ if
there is a bounded linear operator $f_{x}^{\prime }:X\rightarrow Y$ (called
the \emph{derivative} of $f$ at $x$) such that for each $h\in X$ and each
linear continuous functional $y^{\ast }\in Y^{\ast }$ we get
\begin{equation*}
\lim_{t\rightarrow 0}\frac{y^{\ast }\big(f(x+th)-f(x)\big)}{t}=y^{\ast
}\circ f_{x}^{\prime }(h).
\end{equation*}%
If
\begin{equation*}
\lim_{h\rightarrow 0}\frac{\Vert f(x+h)-f(x)-f_{x}^{\prime }(h)\Vert }{\Vert
h\Vert }=0,
\end{equation*}%
then $f$ is \emph{Fr\'{e}chet differentiable} at $x$.

The derivative $f^{\prime }_x$ belongs to the Banach space $L(X,Y)$ of all
bounded linear operators from $X$ to $Y$, endowed with the operator norm $%
\|T\|=\sup_{\|x\|\le 1}\|T(x)\|$.

The following two propositions are known and we present their short proofs
for completeness.

\begin{proposition}
Let $X,Y$ be Banach spaces and $U\subset X$ be an open subset. A function $%
f:U\to Y$ is Lipschitz if $f$ is weakly G\^ateaux differentiable at each
point of $U$ and the derivative map $f^{\prime }:U\to L(X,Y)$, $f^{\prime
}:x\mapsto f^{\prime }_x$, is bounded. In this case $\mathrm{Lip}_1(f)\le
\|f^{\prime }\|_\infty=\sup_{x\in U}\|f^{\prime }_x\|$.
\end{proposition}

\begin{proof}
Let $L=\|f^{\prime }\|_\infty$. The inequality $\mathrm{Lip}_1(f)\le
L=\|f^{\prime }\|_\infty$ will follow as soon as we check that
\begin{equation*}
\|f(x+h)-f(x)\|\le L\|h\|
\end{equation*}
for any $x\in U$ and $h\in X$ with $[x,x+h]\subset U$. Using the Hahn-Banach
Theorem, find a linear continuous functional $y^*\in Y^*$ with unit norm $%
\|y^*\|=1$ such that $y^*\big(f(x+h)-f(x)\big)=\|f(x+h)-f(x)\|$. The weak
G\^ateaux differentiability of $f $ implies that the function
\begin{equation*}
g:[0,1]\to \mathbb{C},\;\;g:t\mapsto y^*(f(x+th)-f(x))
\end{equation*}%
is differentiable and $g^{\prime}(t)=y^*\circ f^{\prime }_{x+th}(h)$ for
each $t\in[0,1]$. Then
\begin{equation*}
\|g^{\prime}\|_\infty\le \|y^*\|\cdot\|f^{\prime }_{x+th}\|\cdot\|h\|\le
1\cdot\|f^{\prime }\|_\infty\cdot\|h\|=L\cdot\|h\|
\end{equation*}%
and
\begin{equation*}
\|f(x+h)-f(x)\|=|g(1)-g(0)|=\Big|\int_0^1g^{\prime }(t)dt\Big|%
\le\int_0^1|g^{\prime }(t)|dt\le L\|h\|\int_0^1dt=L\|h\|.
\end{equation*}
\end{proof}

\begin{proposition}
\label{diff} Let $X,Y$ be Banach spaces and $U\subset X$ be an open subset.
Assume that a function $f:U\to Y$ is weakly G\^ateaux differentiable at each
point of $U$ and the derivative map $f^{\prime }:U\to L(X,Y)$, $f^{\prime
}:x\mapsto f^{\prime }_x$, is Lipschitz. Then

\begin{enumerate}
\item $f$ is Fr\'echet differentiable at each point of $U$;

\item $f$ is second order Lipschitz with $\mathrm{Lip}_2(f)\le\mathrm{Lip}%
_1(f^{\prime })$.
\end{enumerate}
\end{proposition}

\begin{proof}
Let $L=\mathrm{Lip}_1(f^{\prime })$. The Fr\'echet differentiability of $f$
at a point $x\in U$ will follow as soon as we check that
\begin{equation*}
\|f(x+h)-f(x)-f^{\prime }_x(h)\|\le \tfrac12L\|h\|^2
\end{equation*}%
for each $h\in X$ with $[x,x+h]\subset U$. Using the Hahn-Banach Theorem,
choose a linear continuous functional $y^*\in Y^*$ such that $\|y^*\|=1$ and
$y^*\big(f(x+h)-f(x)-f^{\prime }_x(h)\big)=\|f(x+h)-f(x)-f^{\prime }_x(h)\|$%
. The weak G\^ateaux differentiability of $f$ implies that the function
\begin{equation*}
g:[0,1]\to\mathbb{C},\;\;g:t\mapsto y^*(f(x+th)-tf^{\prime }_x(h)),
\end{equation*}
is differentiable. Moreover, for each $t\in[0,1]$ we get $g^{\prime}(t)=y^*\circ
f^{\prime}_{x+th}(h)-y^*\circ f^{\prime }_x(h)$ and
\begin{equation*}
|g^{\prime}(t)|=|y^*(f^{\prime }_{x+th}(h)-f^{\prime
}_x(h))|\le\|y^*\|\cdot\|f^{\prime }_{x+th}(h)-f^{\prime }_x(h)\|\le
\|f^{\prime }_{x+th}-f^{\prime }_x\|\cdot\|h\|\le \mathrm{Lip}_1(f^{\prime})\cdot\|th\|\cdot\|h\|=tL\|h\|^2.
\end{equation*}
Then
\begin{equation*}
\|f(x+h)-f(x)-f^{\prime }_x(h)\|=|g(1)-g(0)|=\Big|\int_0^1g^{\prime }(t)dt%
\Big|\le\int_0^1|g^{\prime }(t)|dt\le \int_0^1 tL\|h\|^2dt=\frac12L\|h\|^2.
\end{equation*}

To see that $f$ is second order Lipschitz, observe that for each $h\in X$
with $[x-h,x+h]\subset U$ we get
\begin{equation*}
\begin{aligned}
\|f(x+h)-2f(x)+f(x-h)\|&=\|f(x+h)-f(x)-f'_x(h)+f(x-h)-f(x)-f'_x(-h)\|\le\\ &\le\|f(x+h)-f(x)-f'_x(h)\|+\|f(x-h)-f(x)-f'_x(-h)\|\le 2\frac12L\|h\|^2=L\|h\|^2,
\end{aligned}
\end{equation*}
which implies that $\mathrm{Lip}_2(f)\le L=\mathrm{Lip}_1(f^{\prime })$.
\end{proof}

\section{Lipschitz-open maps}

Let $X,Y$ be Banach spaces. A map $f:U\to Y$ defined on an open subset $%
U\subset X$ is called \emph{Lipschitz-open} if there is a positive constant $%
c$ such that for each $x\in X$ and $\varepsilon>0$ with $B_\varepsilon(x)%
\subset U$ we get $B_{c\varepsilon}(f(x))\subset f(B_\varepsilon(x))$.
Observe that a map $f:U\to Y$ is Lipschitz-open if and only if its \emph{%
Lipschitz-open constant}
\begin{equation*}
\mathrm{Lip}_o(f)=\sup\big\{c\in[0,\infty):\forall x\in
U\;\forall\varepsilon>0\;\; B_\varepsilon(x)\subset U\Rightarrow
B_{c\varepsilon}(f(x))\subset f(B_\varepsilon(x))\big\}
\end{equation*}
is strictly positive.

A map $f:U\to Y$ is \emph{locally Lipschitz-open} if each point $x\in U$ has
an open neighborhood $W\subset U$ such that the restriction $f|W:W\to Y$ is
Lipschitz-open.

Observe that a bijective map $f:X\to Y$ between Banach spaces is
Lipschitz-open if and only if the inverse map $f^{-1}:Y\to X$ is Lipschitz.
In this case $\mathrm{Lip}_o(f)=\mathrm{Lip}_1(f^{-1})$.

The following proposition can be derived from Theorem 15.5 of \cite{Deimling}%
.

\begin{proposition}
\label{p:lipo} Let $X,Y$ be Banach spaces. A map $f:U\to Y$ defined on an
open subspace $U$ of $X$ is locally Lipschitz-open if

\begin{enumerate}
\item $f$ is weakly G\^ateaux differentiable and the derivative $f^{\prime
}_x:X\to Y$ is surjective at each point $x\in U$;

\item the derivative $f^{\prime }:U\to L(X,Y)$ is Lipschitz.
\end{enumerate}
\end{proposition}

\section{Main Results}

\begin{theorem}
\label{main} Let $X,Y$ be Banach spaces. A map $f:U\to Y$ defined on an open
subspace $U\subset X$ is uniformly locally convex if

\begin{enumerate}
\item the Banach space $X$ has modulus of convexity of power type 2,

\item $f$ is second order Lipschitz;

\item $f$ is Lipschitz-open.
\end{enumerate}

Moreover, in this case $f$ has local convexity radius $\mathrm{lcr}(f)\ge
8\cdot {\mathrm{Lip}_o(f)\cdot\mathrm{conv}_2(X)}/{\mathrm{Lip}_2(f)}>0.$
\end{theorem}

\begin{proof}
Given any point $x_0\in U$ and a positive $\varepsilon\le8\cdot\mathrm{Lip}%
_o(f)\cdot\mathrm{conv}_2(X) / \mathrm{Lip}_2(f)$ with $B_\varepsilon(x_0)%
\subset U$, we need to prove that the image $f(B_\varepsilon(x_0))$ is
convex. Without loss of generality, $x_0=0$.

\begin{claim}
\label{cl1} For any points $a,b\in f(B_\varepsilon(x_0))$ we get $(a+b)/2\in
f(B_\varepsilon(x_0))$.
\end{claim}

\begin{proof}
Find two points $x,y\in B_\varepsilon(x_0)=B_\varepsilon(0)$ with $a=f(x)$
and $b=f(y)$, and consider the midpoint $z=(x+y)/2$. Observe that the points
$x_\varepsilon=x/\varepsilon$, $y_\varepsilon=y/\varepsilon$, and $%
z_\varepsilon=z/\varepsilon$ have norms $\le 1$.

The definition of the 2-convexity number $\mathrm{conv}_2(X)$ guarantees
that
\begin{equation*}
1-\frac1\varepsilon\|z\|=1-\|z_\varepsilon\|\ge\mathrm{conv}_2(X)\cdot
\|x_\varepsilon-y_\varepsilon\|^2=\frac1{\varepsilon^2}\mathrm{conv}%
_2(X)\|x-y\|^2
\end{equation*}%
and thus
\begin{equation*}
\varepsilon-\|z\|\ge \frac1{\varepsilon}\mathrm{conv}_2(X)\|x-y\|^2.
\end{equation*}
Then $B_\delta(z)\subset B_\varepsilon(x_0)$, where
\begin{equation*}
\delta=\frac1{\varepsilon}\mathrm{conv}_2(X)\|x-y\|^2\ge \frac{\mathrm{Lip}%
_2(f)}{8\mathrm{Lip}_o(f)\cdot\mathrm{conv}_2(X)}\mathrm{conv}_2(X)\|x-y\|^2=%
\frac{\mathrm{Lip}_2(f)}{8\mathrm{Lip}_o(f)}\|x-y\|^2
\end{equation*}
and hence
\begin{equation*}
f(B_\varepsilon(x_0))\supset f(B_\delta(z))\supset B_{\mathrm{Lip}%
_o(f)\delta}(f(z))=B_{\eta}(f(z))
\end{equation*}%
where $\eta=\mathrm{Lip}_o(f)\,\delta=\frac18\mathrm{Lip}_2(f)\|x-y\|^2.$

The definition of the constant $\mathrm{Lip}_2(f)$ implies that for $h=z-x$,
we get
\begin{equation*}
\begin{aligned}
\|(a+b)/2-f(z)\|&=\|(f(x)+f(y))/2-f(z)\|=\frac12\|f(z-h)-2f(z)+f(z+h)\|\le\\
&\le \frac12\mathrm{Lip}_2(f)\|h\|^2=\frac18\mathrm{Lip}_2(f)\|x-y\|^2=\eta
\end{aligned}
\end{equation*}%
and hence $(a+b)/2\in B_\eta(f(z))\subset f(B_\varepsilon(x_0))$.
\end{proof}

Claim~\ref{cl1} implies that the closure $\mathrm{cl}_Y(f(B_%
\varepsilon(x_0)) $ is convex. The Lipschitz-openness of the map $f$ implies
that for any numbers $\delta<\eta<\varepsilon$ we get $\mathrm{cl}%
_Y(f(B_\delta(x_0))\subset f(B_\eta(x_0))$. Then the open set $%
f(B_\varepsilon(x_0))$ is convex, being the union
\begin{equation*}
f\big(B_\varepsilon(x_0)\big)=f\Big(\bigcup_{0<\delta<\varepsilon}B_%
\delta(x_0)\Big)=\bigcup_{0<\delta<\varepsilon}\mathrm{cl}_Y\big(%
f(B_\delta(x_0)\big)
\end{equation*}%
of a linearly ordered chain of convex sets.
\end{proof}

Taking into account that each Hilbert space $X$ has 2-convexity number $%
\mathrm{conv}_2(E)\ge\frac18$, and applying Theorem~\ref{main}, we get:

\begin{corollary}
\label{main-1} Let $Y$ be a Banach space and $U$ be an open subspace of a
Hilbert space $X$. Each Lipschitz-open second order Lipschitz map $%
f:U\rightarrow Y$ is uniformly locally convex and has local convexity radius
$\mathrm{lcr}(f)\geq {\mathrm{Lip}_{o}(f)}/{\mathrm{Lip}_{2}(f)}>0.$
\end{corollary}

Theorem~\ref{main} combined with Propositions~\ref{diff} and \ref{p:lipo}
implies the following two corollaries.

\begin{corollary}
\label{C1} Let $X,Y$ be Banach spaces. A map $f:U\to Y$ defined on an open
subspace $U\subset X$ is uniformly locally convex if

\begin{enumerate}
\item the Banach space $X$ has modulus of convexity of power type 2,

\item $f$ is weakly G\^ateaux differentiable and the derivative $f^{\prime
}:U\to L(X,Y)$ is Lipschitz;

\item $f$ is Lipschitz-open.
\end{enumerate}
\end{corollary}

\begin{corollary}
\label{C2} Let $X,Y$ be Banach spaces. A map $f:U\to Y$ defined on an open
subspace $U\subset X$ is locally convex if

\begin{enumerate}
\item the Banach space $X$ has modulus of convexity of power type 2,

\item $f$ is weakly G\^ateaux differentiable and the derivative $f^{\prime
}:U\to L(X,Y)$ is Lipschitz;

\item for each $x\in U$ the derivative $f^{\prime }_x:X\to Y$ is surjective.
\end{enumerate}
\end{corollary}

\section{An Open Problem}

We do not know if the requirement on the convexity modulus of the Banach
space $X$ is essential in Theorem~\ref{main} and Corollaries~\ref{C1}, \ref%
{C2}.

\begin{problem}
\label{prob1} Assume that $X$ is a Banach space such that any Lipschitz-open
second order Lipschitz map $f:U\to X$ defined on an open subset $U\subset X$
is locally convex. Has $X$ the modulus of convexity of power type 2? Is $X$
(super)reflexive?
\end{problem}

\section{Acknowledgements}

The fourth author (A.P.) is grateful to Prof. L. G\'{o}rniewicz for
invitation to take part in the VI Symposium of Nonlinear Analysis held 7 --
9 September 2011 in Toru\'{n} and for very fruitful discussion and remarks.
He also sincerely thanks Prof. A. Augustynowicz for valuable comments on the
first draft of the papers and mentioning the references related with the
topic studied in the article. Special acknowledgment belongs to the
Scientific and Technological Research Council of Turkey (TUBITAK/NASU-111T558) for a partial support
of   A.K. Prykarpatsky's research. 


\newpage

\begin{thebibliography}{99}
\bibitem{Aug} Augustynowicz A., Dzedzej Z., Gelman B.D. \textit{The solution
set to BVP for some functional-differential inclusions}, Set-Valued
Analysis. \textbf{6} (1998) 257--263.

\bibitem{BP} Blackmore D., Prykarpatsky A.K. \textit{A solution set analysis
of a nonlinear operator equation using a Leray Schauder type fixed point
approach}. Topology. \textbf{48} (2009) 182--185.

\bibitem{BGHV} Borwein J., Guirao A., H\'{a}jek P., ~Vanderwerff J. \emph{%
Uniformly convex functions on Banach spaces}. Proc. Amer. Math. Soc. \textbf{%
137} (2009), no. 3, 1081--1091.

\bibitem{Deimling} Deimling K. \emph{Nonlinear Functional Analysis},
Springer-Verlag, Berlin, 1985.

\bibitem{DGZ} Deville R, Godefroy G., Zizler V. \emph{Smoothness and
renormings in Banach spaces}, Longman Scientific \&\ Technical, Harlow, New
York, 1993.

\bibitem{VL} DeVore R., Lorentz G. \emph{Constructive Approximation},
Springer-Verlag, Berlin, 1993.

\bibitem{Ge1} Goebel K., Kirk W.A. \textit{Topics in metric fixed point
theory}. Cambridge University Press, Cambridge, 1990.

\bibitem{Ge2} Goebel K. T\textit{wierdzenia o punktach sta\l ych}. \textit{%
Wyk\l ady}. Wydawnictwo Uniwersytetu Marii-Curie Sk\l odowskiej, Lublin,
2005.

\bibitem{Go} G\'{o}rniewicz L. \textit{Topological fixed point theory of
muipltivalued mappings.} Kluwer, Dordrecht, 1999.

\bibitem{GH} Guirao A.J., H\'{a}jek P. \emph{On the moduli of convexity}.
Proc. Amer. Math. Soc. \textbf{135} (2007), no. 10, 3233--3240.

\bibitem{HMZ} H\'{a}jek P., Montesinos V., Zizler V. \emph{Geometry and
Gateaux smoothness in separable Banach spaces}, preprint (available at
http://www.math.cas.cz/preprint/pre-234.pdf).

\bibitem{Ho} H\"{o}rmander L. \textit{Sur la fonction d'applui des ensembles
convexes dans une espace localement convexe.} Arkiv Math. \textbf{3} (1955),
no.2, 180--186.

\bibitem{KZ} Krasnoselsky M.A., Zabreyko P.P. \textit{Geometric methods of
nonolinear analysis}. "Nauka" Publisher, Moscow, 1975 (in Russian).

\bibitem{LPT} Lajara S., Pallar\'{e}s A. J., Troyanski S. \emph{Moduli of
convexity and smoothness of reflexive subspaces of $L\sp1$}, J. Funct. Anal.
\textbf{261} (2011), no.11, 3211--3225.

\bibitem{LT2} Lindenstrauss J., Tzafriri L. \emph{Classical Banach spaces.
II. Function spaces}, Springer-Verlag, Berlin-New York, 1979.

\bibitem{Li} Linke Y.E. \textit{Application of Michael's theorem and its
converse to sublinear operators.} Mathematical Notes. \textbf{52}, (1993) 1,
680-686


\bibitem{Nir} Nirenberg L. \textit{Topics in Nonlinear Functional Analysis}.
AMS Publisher, 1974

\bibitem{PBPP} Prykarpatska N. K., Blackmore D. L., Prykarpatsky A.K.,
Pytel-Kudela M. \emph{On the infimum-type extremality solutions to
Hamilton-Jacobi equations, their regularity properties, and some
generalizations}, Miskolc Math. Notes, \textbf{4} (2003), no.2, 153--176.

\bibitem{Pr1} Prykarpatsky A.K. \ \textit{An infinite dimensional
Borsuk-Ulam type generalization of the Leray-Schauder fixed point theorem
and some applications.} Ukr. Math. Zh. \textbf{60} (2008), no.1, 114--120.

\bibitem{Pr} Prykarpatsky A.K. \textit{A Borsuk Ulam type generalization of
the Leray Schauder fixed point theorem}. Preprint ICTP, IC/2007/028,
Trieste, Italy, 2007.

\bibitem{SPS} Samoilenko A.M., Prykarpats'kyi A.K., Samoilenko V.H. \textit{%
Lyapunov--Schmidt approach to studying homoclinic splitting in weakly
perturbed Lagrangian and Hamiltonian systems.} Ukr. Mat. Zh. \textbf{55}
(2003), no.1, 82--92.

\bibitem{Sch} Schwartz J.T. \textit{Nonlinear functional analysis.}
Gordonand Breach Science Publisher, NY, 1969.
\end{thebibliography}
\end{document}